\documentclass[letterpaper, 10 pt, conference]{ieeeconf}  %
\IEEEoverridecommandlockouts                              %
\newtheorem{theorem}{Theorem}

\newtheorem{prop}{\bfseries Proposition}

\newtheorem{remark}{\bfseries Remark}

\newtheorem{lemma}{\bfseries Lemma}
\newtheorem{assump}{\bfseries Assumption}
\newtheorem{corollary}{Corollary}
\usepackage{graphics} 
\usepackage{epsfig} 
\usepackage{mathptmx} 
\usepackage{times} 
\usepackage{amsmath} 
\usepackage{amssymb}  
\usepackage{mathtools}

\usepackage{algorithm}
\usepackage{algorithmic}
\usepackage{float} 
\usepackage{capt-of}
\usepackage{tcolorbox}
\usepackage{tabularx}
\usepackage{array}
\usepackage{xcolor}
\usepackage[caption=false, font=footnotesize]{subfig}

\usepackage{verbatim}
\usepackage{cite}

\makeatletter
\let\NAT@parse\undefined
\makeatother

\usepackage[
    colorlinks=true,
    linkcolor=teal,
    citecolor=blue,
    urlcolor=blue
]{hyperref}

\title{\LARGE \bf
Network Epidemic Control via Model Predictive Control: Extended Version
}

\author{Mahtab~Talaei, Alex~Olshevsky, Laura~F.~White, and Ioannis~Ch.~Paschalidis
\thanks{M. Talaei, A. Olshevsky, and I.Ch. Paschalidis are with the Department of Electrical and Computer Engineering, and the Division of Systems Engineering, Boston University, Boston, MA, USA.  E-mails: {\tt\small \{mtalaei, alexols, yannisp\}@bu.edu}. I.Ch. Paschalidis is also affiliated with the Department of Biomedical Engineering and the Faculty for Computing \& Data Sciences, Boston University, Boston, MA, USA.}
\thanks{L.F. White is with the School of Public Health, Boston University, Boston, MA, USA, e-mail: {\tt\small lfwhite@bu.edu}.}
}

\begin{document}

\maketitle
\thispagestyle{empty}
\pagestyle{empty}


\begin{abstract}
Balancing the societal costs of non-pharmaceutical interventions with epidemic suppression requires adaptive feedback control. 
Rather than relying on state-dependent operational caps, we formulate an infinite-horizon optimal control problem for a networked SIQR model that strictly enforces suppression via a hard spectral constraint on the transmission dynamics. We derive a safety-critical Model Predictive Control (MPC) approximation that embeds this spectral certificate stage-wise, yielding a tunable exponential decay rate. Furthermore, we construct a terminal set ensuring recursive feasibility and a feasible continuation that decays globally, proving positive invariance directly via the physical depletion of susceptibles rather than standard quadratic Lyapunov functions. To handle prediction uncertainty, we develop a robust counterpart that replaces nominal constraints by upper-envelope versions, recovering recursive feasibility and finite-horizon realized decay. We conclude by validating our approaches using simulation studies that leverage public data from counties in the state of Massachusetts.
\end{abstract}


\section{INTRODUCTION}

Non-pharmaceutical interventions (NPIs), including isolation policies, distancing, and mobility restrictions, can substantially reduce transmission during emerging outbreaks~\cite{flaxman2020estimating, hsiang2020effect} but impose large societal and economic costs~\cite{thunstrom2020benefits}. Designing intervention policies that balance suppression with societal burden under time-varying conditions (seasonality, behavioral shifts, viral evolution~\cite{davies2021estimated}) has traditionally been approached via open-loop optimal control on compartmental models~\cite{ma2023optimal, lenhart2007optimal}. Open-loop policies, however, are sensitive to model mismatch: transmission parameters are uncertain~\cite{barnett2023epidemic}, observability is delayed, and policies must be revised as new data arrive. These issues are amplified in networked settings, where mobility coupling yields spatial heterogeneity and necessitates coordinated interventions~\cite{nowzari2016analysis}, motivating feedback strategies that adapt policies in real time.

Model Predictive Control (MPC) naturally provides such feedback~\cite{rawlings2017model} and has been actively explored for epidemic mitigation: robust MPC with bounded parameter uncertainty for networked SIS dynamics~\cite{watkins2019robust} and COVID-19 models~\cite{kohler2021robust}, scenario-based MPC that hedges across forecast realizations~\cite{herceg2025scenario}, 
stochastic and probabilistic MPC formulations~\cite{armaou2022designing}, nonlinear MPC with logic constraints for hospitalization capacity~\cite{peni2020nonlinear}, multi-region MPC for mobility-coupled networks~\cite{carli2020model}, and terminal-ingredient-free schemes that exploit viability-kernel structure of the SEIR equilibrium~\cite{esterhuizen2024mpc}. 
A parallel, model-agnostic line uses reinforcement learning 
to tune intervention policies directly from simulated 
data~\cite{gemignani2026real}.

Despite their methodological diversity, these works share a structural feature: \emph{safety is enforced as a 
state-dependent operational cap} (typically on hospital or 
ICU occupancy), \emph{or as a robust margin on the cost}, rather than as an explicit suppression certificate on the infection dynamics themselves. Operational caps do not preclude infection growth. Prevalence can grow while remaining below capacity and cost-based objectives do not, on their own, certify a guaranteed decay rate. 
We therefore start from a different premise: we formulate an 
infinite-horizon optimal control problem in which suppression 
is imposed as a hard spectral constraint on the 
infected-subsystem dynamics, with a tunable exponential decay 
rate $\alpha$, while isolation burden enters only through the 
cost. The proposed MPC is then derived as the tractable 
finite-horizon approximation of this formulation.

\noindent\textbf{Contributions.} 
Building on the networked SIQR model of~\cite{talaei2025network}, we make three contributions:
\textbf{(i)}~A stage-wise spectral suppression certificate embedded as a hard constraint in receding-horizon control. The certificate encodes infection decay at the level of the dynamics matrix independently of current prevalence, contrasting with state-dependent operational caps.
\textbf{(ii)}~Nominal recursive feasibility, finite-horizon closed-loop decay, and the existence of a feasible continuation achieving global exponential decay under this certificate; and a robust counterpart in which the nominal stage-wise and terminal constraints are replaced by upper-envelope versions, recovering recursive feasibility and finite-horizon realized decay under a strict relaxation of the exact-prediction hypothesis. The terminal-set construction follows standard MPC methodology~\cite{mayne2000constrained}, but our positive-invariance proof exploits the structural depletion of susceptibles rather than quadratic Lyapunov certificates.
\textbf{(iii)}~Numerical validation on a 14-county Massachusetts mobility network: under administrative rate limits and a variant-induced surge, the controller preserves the design decay rate, where state-dependent or smoothly-penalized formulations do not.

\noindent\textbf{Organization.} Sec.~\ref{sec:problem} introduces the model, spectral decay condition, and MPC formulation. Sec.~\ref{sec:main_results} establishes recursive feasibility, stability, and a robust extension under bounded forecast error. Sec.~\ref{sec:numerical_results} presents numerical results. Sec.~V concludes.

\section{PROBLEM FORMULATION AND MPC DESIGN}
\label{sec:problem}
\subsection{Networked SIQR Dynamics}\label{subsec:network_dynamics}
Consider a network of $n$ regions. Let $s_i(t)$, $x_i^{\mathrm{a}}(t)$, $x_i^{\mathrm{s}}(t)$, and $k_i(t)$ denote the susceptible, asymptomatic, symptomatic, and isolated population proportions at node $i$. The continuous-time SIQR dynamics with time-varying transmission are~\cite{talaei2025network}:
\begin{subequations}
\label{eq:SIQR-continuous}
\begin{align}
\dot{s}_i(t)
&= - s_i(t) \sum_{j=1}^{n} a_{ij}
\Bigl( \beta^{\mathrm{a}}(t) x_j^{\mathrm{a}}(t) + \beta^{\mathrm{s}}(t) x_j^{\mathrm{s}}(t) \Bigr),
\label{eq:SIQR-s} \\
\dot{x}_i^{\mathrm{a}}(t)
&= -\dot{s}_i(t) - \bigl( \epsilon + r^{\mathrm{a}} + q_i^{\mathrm{a}}(t) \bigr) x_i^{\mathrm{a}}(t),
\label{eq:SIQR-xa} \\
\dot{x}_i^{\mathrm{s}}(t)
&= \epsilon x_i^{\mathrm{a}}(t)
- \bigl( r^{\mathrm{s}} + q_i^{\mathrm{s}}(t) \bigr) x_i^{\mathrm{s}}(t),
\label{eq:SIQR-xs} \\
\dot{k}_i(t)
&= q_i^{\mathrm{a}}(t) x_i^{\mathrm{a}}(t) + q_i^{\mathrm{s}}(t) x_i^{\mathrm{s}}(t) - r^{\mathrm{q}} k_i(t).
\label{eq:SIQR-k}
\end{align}
\end{subequations}

\noindent\textbf{Notation \& Parameters.} Transmission rates $\beta(t) \coloneqq (\beta^{\mathrm{a}}(t), \beta^{\mathrm{s}}(t)) \in \mathbb{R}_{>0}^2$ are bounded by $\beta_{\max}$ for all $t \ge 0$. We define symptom onset rate $\epsilon > 0$, recovery rates $r^{\mathrm{a}}, r^{\mathrm{s}}, r^{\mathrm{q}} > 0$, and the mobility-derived infection flow matrix $A = [a_{ij}] \in \mathbb{R}_{\ge 0}^{n \times n}$. Control inputs (isolation rates) $q_i^{\mathrm{a}}(t), q_i^{\mathrm{s}}(t) \in [0, B]$ are aggregated as $q(t) \coloneqq (q^{\mathrm{a}}(t), q^{\mathrm{s}}(t)) \in \mathbb{R}^{2n}$. The full state $x(t) = (s, x^{\mathrm{a}}, x^{\mathrm{s}}, k) \in \mathbb{R}^{4n}$ satisfies $\dot{x}(t) = f(t, x(t), q(t))$ per~\eqref{eq:SIQR-continuous}.

\subsection{Spectral Condition for Exponential Decay}
\label{subsec:spectral_decay}

The \emph{infected state} $y(t) \coloneqq (x^{\mathrm{a}}(t), x^{\mathrm{s}}(t)) \in \mathbb{R}^{2n}_{\geq 0}$ satisfies the linear time-varying system $\dot{y}(t) = M\bigl(s(t), q(t) \mid \beta(t)\bigr) y(t)$, where the $2n \times 2n$ \emph{infected-subsystem matrix} is:
\begin{equation} \label{eq:M-matrix}
  M(s, q \mid \beta)
  \coloneqq
  \begin{pmatrix}
    E(s, q^{\mathrm{a}}, \beta^{\mathrm{a}})
    &
    \beta^{\mathrm{s}} \operatorname{diag}(s) A
    \\[4pt]
    \epsilon I_n
    &
    -\bigl(r^{\mathrm{s}} I_n + \operatorname{diag}(q^{\mathrm{s}})\bigr)
  \end{pmatrix},
\end{equation}
with upper-left block $E(s, q^{\mathrm{a}}, \beta^{\mathrm{a}}) = \beta^{\mathrm{a}} \operatorname{diag}(s) A - (\epsilon + r^{\mathrm{a}}) I_n - \operatorname{diag}(q^{\mathrm{a}})$.

\noindent\textbf{Structural Properties.} Because $A \ge 0$ and $\beta > 0$, $M$ is a Metzler matrix (nonnegative off-diagonals). Its upper block terms depend linearly on $s$ and $\beta$ with nonnegative coefficients, so $M$ is entry-wise non-decreasing in $s$, $\beta^{\mathrm{a}}$, and $\beta^{\mathrm{s}}$. Since $\dot{s}_i \le 0$, $M(s(t), q \mid \beta)$ is entry-wise nonincreasing in $t$ for fixed $q$ and $\beta$. Because of the Metzler structure, the stability of the infected subsystem is governed by the spectral abscissa $\lambda_{\max}(X) \coloneqq \max\{\Re(\lambda) : \lambda \in \sigma(X)\}$ of the matrix~$M$, where $\sigma(X)$ denotes the spectrum (set of eigenvalues) of~$X$.

\noindent\textbf{Exponential Decay Under Per-Step Constraints.}
In MPC, the isolation policies and parameter forecasts are updated in discrete steps. Fix a finite horizon $T > t_0$ and let $M_{\mathrm{tot}} \in \mathbb{N}$ satisfy $t_{M_{\mathrm{tot}}} \leq T < t_{M_{\mathrm{tot}}+1}$.

\begin{assump}
\label{assum:piecewise_inputs}
On each sampling interval $[t_m, t_{m+1})$ with $\Delta t = t_{m+1} - t_m$, the control $q(t) = q_m$ and transmission $\beta(t) = \beta_m$ are constant, with $\beta_m$ known at $t_m$.
\end{assump}

\begin{assump}\label{assum:coupling}
The directed graph of $A$ is strongly connected, and $s_i(t) > 0$ for all $i$ and $t \in [t_0, T]$.
\end{assump}

Let $M_m \coloneqq M(s(t_m), q_m \mid \beta_m)$. Because $s(t) \le s(t_m)$ for all $t \in [t_m, t_{m+1})$, we have $M\bigl(s(t), q_m \mid \beta_m\bigr) \le M_m$. This inter-sample monotonicity is crucial for continuous-time decay guarantees.

\begin{prop}
\label{prop:decay_eig_clean}
Under Ass.~\ref{assum:piecewise_inputs} and~\ref{assum:coupling},
suppose $\alpha > 0$ satisfies $\lambda_{\max}(M_m) \leq -\alpha$ for
$m \in \{0, \ldots, M_{\mathrm{tot}}\}$, and let $v_m > 0$ be the left
Perron eigenvector of $M_m$ normalized so that $\|v_m\|_1 = 1$. Then
the Perron-weighted norm contracts on each sampling interval,
\begin{equation}\label{eq:per_interval}
  v_m^\top y(t) \;\leq\; e^{-\alpha(t - t_m)}\, v_m^\top y(t_m),
  \qquad t \in [t_m, t_{m+1}),
\end{equation}
and consequently
\begin{equation}\label{eq:decay_T}
  \|y(t)\|_1
  \;\leq\;
  C_T\, e^{-\alpha(t - t_0)}\, \|y(t_0)\|_1,
  \qquad \forall t \in [t_0, T],
\end{equation}
with $C_T \coloneqq (\bar{v}_T/\underline{v}_T)^{M_{\mathrm{tot}}+1}$,
where $\bar{v}_T \coloneqq \max_{m,i} (v_m)_i$ and
$\underline{v}_T \coloneqq \min_{m,i} (v_m)_i > 0$.
\end{prop}
\begin{proof}
Let $t \in [t_m, t_{m+1})$. By Ass.~\ref{assum:coupling}, $M_m$ is irreducible Metzler. Let $v_m > 0$ be the left Perron eigenvector satisfying $v_m^\top M_m = \lambda_{\max}(M_m) v_m^\top$ with $\|v_m\|_1 = 1$. Let $V_m(t) \coloneqq v_m^\top y(t)$. By monotonicity, $\dot{V}_m(t) = v_m^\top M(s(t), q_m \mid \beta_m) y(t) \le v_m^\top M_m y(t) \le -\alpha V_m(t)$. Grönwall's inequality yields~\eqref{eq:per_interval}. Since $v_m$ is normalized and entrywise positive, $\underline{v}_T \|y(t)\|_1 \leq V_m(t) \leq \bar{v}_T \|y(t)\|_1$, so $\|y(t)\|_1 \le (\bar{v}_T/\underline{v}_T)\, e^{-\alpha(t-t_m)} \|y(t_m)\|_1$. Taking the left-limit $t\uparrow t_m$ in the estimate
on $[t_{m-1},t_m)$ gives $\|y(t_m)\|_1 \le (\bar{v}_T/\underline v_T)\, e^{-\alpha\Delta t} \|y(t_{m-1})\|_1$; iterating yields $\|y(t_m)\|_1 \le (\bar{v}_T/\underline v_T)^m
e^{-\alpha(t_m-t_0)}\|y(t_0)\|_1$ for all $m\in\{0,\dots,M_{\mathrm{tot}}\}$. Substituting into the inter-sample bound for $t\in[t_m,t_{m+1})$ gives
$\|y(t)\|_1 \le (\bar{v}_T/\underline v_T)^{m+1}
e^{-\alpha(t-t_0)}\|y(t_0)\|_1$, and $m\le M_{\mathrm{tot}}$
yields~\eqref{eq:decay_T}.
\end{proof}

\begin{remark}[Structure of $C_T$]
\label{rem:CT_structure}
The constant $C_T = (\bar{v}_T/\underline{v}_T)^{M_{\mathrm{tot}}+1}$ accounts for the potential misalignment of Perron eigenvectors when the matrices $M_m$ switch at sampling instants. Because $\underline{v}_T < \bar{v}_T$ in general, $C_T$ may grow with $M_{\mathrm{tot}}$ (and hence with $T$). Thus, Prop.~\ref{prop:decay_eig_clean} provides a \emph{finite-horizon} guarantee on $[t_0, T]$, justifying the per-step MPC constraint~\cite{talaei2025network}. Global exponential decay ($t \to \infty$) is established separately in Theorem~\ref{thm:global_exp_decay} by exhibiting a feasible continuation with a single terminal Lyapunov vector.
\end{remark}

\subsection{Optimal Control Formulation}

The infinite-horizon optimal control problem is formulated as:
\begin{equation} \label{eq:OCP-infinite}
\begin{aligned}
  \min_{q(\cdot)}
  &\int_{0}^{\infty} \left[ 
      \sum_{i=1}^{n} w_i k_i(t) 
      + \frac{\rho}{2} \sum_{i=1}^{n} w_i \bigl((q_i^{\mathrm{a}}(t))^2 + (q_i^{\mathrm{s}}(t))^2\bigr)
    \right] dt \\
  \text{s.t.}\quad
  & \dot{x}(t) = f\bigl(t, x(t), q(t)\bigr), \quad x(0) = x_0, \\
  & \lambda_{\max}\bigl(M(s(t), q(t) \mid \beta(t))\bigr) \leq -\alpha, \quad \forall t \geq 0, \\
  & 0 \leq q_i^{\mathrm{a}}(t), q_i^{\mathrm{s}}(t) \leq B, \quad \forall i \in \{1,\ldots,n\},\ \forall t \geq 0,
\end{aligned}
\end{equation}
where $w_i \coloneqq N_i / \sum_{j} N_j$ weights the normalized isolation burden by regional population, and $\rho > 0$ penalizes control effort. For admissible policies $q(t) \in [0,B]^{2n}$, the spectral constraint ensures exponential infection decay (Sec.~\ref{subsec:stability}). Consequently, $k(t) \to 0$ via~\eqref{eq:SIQR-k}. The control effort term is finite only if susceptible depletion eventually allows $q = 0$; in practice, the 
controller terminates once infections are negligible, so the realized cost is always finite. Problem~\eqref{eq:OCP-infinite} establishes the baseline costs and constraints inherited by the MPC in Sec.~\ref{subsec:MPC_formulation}.

\noindent\textbf{Computational Challenge.} Problem~\eqref{eq:OCP-infinite} is infinite-dimensional and non-convex. Although $\lambda_{\max}(\cdot)$ is convex in the diagonal control entries $q$ for fixed $s$ and $\beta$~\cite{1981:Cohen}, the trajectory $s(t)$ depends nonlinearly on $q$ through~\eqref{eq:SIQR-continuous}. This coupled non-convexity makes classical infinite-horizon optimization intractable for moderate network sizes, motivating the discrete-time finite-horizon approximation derived below.


\subsection{Finite-Horizon MPC Formulation}
\label{subsec:MPC_formulation}

To obtain a computationally tractable controller, we discretize \eqref{eq:OCP-infinite} using a sampling period $\Delta t > 0$ and adopt a receding-horizon strategy. At each sampling time $t_m$, we solve a finite-horizon approximation over $[t_m, t_m + H\Delta t)$ under piecewise-constant inputs (Ass.~\ref{assum:piecewise_inputs}), where $H \in \mathbb{N}$ is the prediction horizon. Let $q_j \coloneqq (q^{\mathrm{a}}_j, q^{\mathrm{s}}_j)$ denote the constant control vector applied over the prediction interval $[t_{m,j}, t_{m,j+1})$, where $t_{m,j} \coloneqq t_m + j\Delta t$. Let $\hat{x}_j$ denote the predicted state vector at $t_{m,j}$, with $\hat{x}_0 = x(t_m)$. Only the first control action $q_0$ is applied; the state is measured at $t_{m+1}$, and the process repeats.

To enable preemptive control, the MPC utilizes a transmission forecast profile $\Gamma_m \coloneqq \{\hat{\beta}_{m,j}\}_{j=0}^{H-1}$, available at time $t_m$. Anticipated surges due to seasonality, 
holidays, or policy changes are reflected in $\Gamma_m$.

We define the stage cost penalizing the isolated population and control effort as
$ \ell(\hat{x}_j, q_j) \coloneqq
  \sum_{i=1}^{n} w_i \hat{k}_{i,j} +
  \frac{\rho}{2} \sum_{i=1}^{n} w_i \bigl((q^{\mathrm{a}}_{i,j})^2 + (q^{\mathrm{s}}_{i,j})^2\bigr),$
where $\hat{k}_{i,j}$ is the $i$-th component of the isolated population in $\hat{x}_j$. The finite-horizon cost is therefore approximated by
\begin{equation} \label{eq:MPC-cost-def}
  J_m\bigl(\{q_j\}_{j=0}^{H-1}\bigr) 
  \;\coloneqq\; 
  \sum_{j=0}^{H-1} \ell(\hat{x}_j, q_j)\,\Delta t
  \;+\; V_f(\hat{x}_H).
\end{equation}
Here, term $\ell(\cdot)\Delta t$ represents the accumulated cost over one sampling interval. The terminal cost $V_f(\hat{x}_H) = \frac{1}{2} \hat{x}_H^\top P \hat{x}_H$ ($P \succeq 0$) penalizes residual infections beyond the prediction window. While $V_f$ improves performance, system stability is strictly enforced by the spectral constraints rather than the value function.

The MPC problem solved at time $t_m$ is:
\begin{subequations}\label{eq:MPC-problem}
\begin{align}
  \min_{\{q_j\}_{j=0}^{H-1}} 
  & J_m\bigl(\{q_j\}_{j=0}^{H-1}\bigr)
    \label{eq:MPC-obj} \\
  \text{s.t.}\
  & \hat{x}_{j+1} = \Psi(\hat{x}_j, q_j \mid \hat{\beta}_{m,j}),
    \quad j = 0,\ldots,H-1,
    \label{eq:MPC-dynamics} \\
  & \lambda_{\max}\!\bigl(M(\hat{s}_j, q_j \mid \hat{\beta}_{m,j})\bigr)
    \leq -\alpha,
    \; j = 0,\ldots,H-1,
    \label{eq:MPC-eig-constraint} \\
  & 0 \;\leq\; q_{i,j}^{\mathrm{a}}, q_{i,j}^{\mathrm{s}} \;\leq\; B,
    \quad \forall i, \, j = 0,\ldots,H-1,
    \label{eq:MPC-bounds} \\
  & \hat{x}_0 \;=\; x(t_m),
    \label{eq:MPC-init} \\
  & \hat{x}_H \;\in\; \mathcal{X}_f.
    \label{eq:MPC-terminal}
\end{align}
\end{subequations}

\noindent\textbf{Constraint Details.}
$\Psi(\cdot \mid \hat{\beta})$ in~\eqref{eq:MPC-dynamics} denotes the numerical integration of the continuous dynamics~\eqref{eq:SIQR-continuous} over $\Delta t$. Although~\eqref{eq:MPC-eig-constraint} is enforced only at discrete steps $t_{m,j}$, the monotonicity $s(t) \leq \hat{s}_j$ ensures continuous satisfaction throughout $[t_{m,j}, t_{m,j+1})$ (Prop.~\ref{prop:decay_eig_clean}). The terminal set $\mathcal{X}_f \subset \mathbb{R}^{4n}$ in~\eqref{eq:MPC-terminal} contains all states stabilizable by a constant safe bounding policy and is defined in Sec.~\ref{subsec:terminal_set}.

\begin{remark}[Problem structure]
\label{rem:problem_structure}
Problem~\eqref{eq:MPC-problem} is a finite-dimensional NLP in $2nH$ decision variables. It is non-convex due to bilinear infection dynamics and the dependence of the predicted susceptible profile $\hat{s}_j$ on earlier controls. At $j=0$, however, $\hat{s}_0 = s(t_m)$ is fixed, making the initial constraint convex. Despite this non-convexity, recursive feasibility (Theorem~\ref{thm:rec_feasibility}) ensures the shifted previous solution appended with the terminal policy is a strictly feasible warm start. Thus, any feasible solution returned by a local solver inherits the safety guarantees (Sec.~\ref{subsec:computation}).
\end{remark}


\subsection{Terminal Set Design and Verification}
\label{subsec:terminal_set}

This section defines a terminal set $\mathcal{X}_f$ and an admissible terminal control law $\kappa_f(\cdot)$ for~\eqref{eq:MPC-terminal}. While the reliance on a terminal set to ensure recursive feasibility follows standard MPC methodology~\cite{mayne2000constrained}, our construction is a tailored application leveraging the specific properties of the SIQR network. Specifically, positive invariance of $\mathcal{X}_f$ follows directly from the physical depletion of susceptibles ($\dot s \le 0$) and the monotonicity of the Metzler matrix $M$, bypassing the need for quadratic Lyapunov or LMI certificates typical in robust MPC.

\noindent\textbf{Terminal Control Law \& Terminal Set Definition.}
We employ a constant terminal policy corresponding to maximum isolation effort across all regions, defined as $\kappa_f(x) \coloneqq q_{\text{safe}} = (B, \dots, B)^\top \in \mathbb{R}^{2n}$. To ensure robustness to the worst-case transmission scenario, the terminal set $\mathcal{X}_f$ is defined as the set of states where the spectral decay constraint holds under $q_{\text{safe}}$ and the global upper bound $\beta_{\max}$:
\begin{equation}\label{eq:terminal_set_def}
  \mathcal{X}_f
  \;\coloneqq\;
  \Bigl\{\, x \in \mathbb{R}^{4n}
  : \lambda_{\max}\!\bigl(M(s(x), q_{\text{safe}} \mid \beta_{\max})\bigr) \leq -\alpha \Bigr\},
\end{equation}
where $s(x)$ is the susceptible vector of state $x$. Geometrically, if the control bound $B$ is sufficient to stabilize the worst-case scenario at full susceptibility (i.e., $s(x) = \mathbf{1}_n$), then $\mathcal{X}_f$ spans the entire admissible state space. Otherwise, it defines a safe invariant region where susceptibility has depleted enough to make containment possible.

\noindent\textbf{Positive Invariance.}
To establish invariance, the discrete-time numerical integration must preserve the natural monotonicity of the susceptible population.

\begin{assump}[Discretization consistency]
\label{assum:discretization}
The numerical integration map $\Psi$ in~\eqref{eq:MPC-dynamics} preserves susceptible monotonicity:
$s(\Psi(x, q \mid \beta)) \leq s(x)$ component-wise for all admissible $x$, $q$, and $\beta$.
\end{assump}
This naturally holds for standard ODE solvers (e.g., Runge--Kutta) with sufficiently small step sizes since $\dot{s}_i \leq 0$.

\begin{lemma}[Terminal set invariance]
\label{lem:terminal_invariance}
Under Ass.~\ref{assum:discretization}, the set $\mathcal{X}_f$ is positively invariant under the terminal 
control $\kappa_f$ for any $\beta \leq \beta_{\max}$:
\begin{equation}\label{eq:terminal_invariance_dt}
  x \in \mathcal{X}_f \;\;\text{and}\;\; 
  \beta \leq \beta_{\max}
  \; \Longrightarrow\;
  \Psi(x, q_{\text{safe}} \mid \beta) \in \mathcal{X}_f.
\end{equation}
\end{lemma}
\begin{proof}
Let $x \in \mathcal{X}_f$ and $x^+ = \Psi(x, q_{\text{safe}} \mid \beta)$. By Ass.~\ref{assum:discretization}, $s(x^+) \leq s(x)$ component-wise. Since $M$ is entry-wise non-decreasing in $s$ (Sec.~\ref{subsec:spectral_decay}), we have $M(s(x^+), q_{\text{safe}} \mid \beta_{\max}) \leq M(s(x), q_{\text{safe}} \mid \beta_{\max})$. By monotonicity of $\lambda_{\max}(\cdot)$ for Metzler matrices~\cite{berman1994nonnegative} and the fact that $x \in \mathcal{X}_f$, it follows that $\lambda_{\max}\bigl(M(s(x^{+}), q_{\text{safe}} \mid \beta_{\max})\bigr) \le \lambda_{\max}\bigl(M(s(x), q_{\text{safe}} \mid \beta_{\max})\bigr) \le -\alpha$. Hence, $x^+ \in \mathcal{X}_f$.
\end{proof}

\section{THEORETICAL GUARANTEES} \label{sec:main_results}
To clarify the scope of our theoretical results, we decompose the MPC guarantees into three layers requiring progressively stronger hypotheses. \textbf{(i) Per-step decay (unconditional):} Feasibility of~\eqref{eq:MPC-problem} at $t_m$ implies $\lambda_{\max}(M(s(t_m), q^*_{m,0} \mid \beta(t_m))) \leq -\alpha$ at the applied control. Combined with the true plant's natural depletion ($\dot{s}_i \leq 0$), this ensures strict exponential decay of infections on $[t_m, t_{m+1})$ \emph{regardless of forecast quality over the rest of the horizon}. \textbf{(ii) Recursive feasibility and global decay (nominal):} Theorems~\ref{thm:rec_feasibility}--\ref{thm:global_exp_decay} establish closed-loop feasibility and stability under exact prediction matching (Ass.~\ref{assum:forecast_consistency}). This layer provides standard nominal MPC guarantees; we explicitly quantify its robustness to forecast errors in Sec.~\ref{subsec:robustness}. \textbf{(iii) Cost optimality:} Forecast quality strictly affects the objective value~\eqref{eq:MPC-obj}; degraded forecasts yield suboptimal but \emph{still safe} controls via Layer (i). In short: Layer (i) uses current measurements for unconditional safety; Layer (ii) uses predictions to guarantee a feasible continuation; Layer (iii) uses predictions to score that continuation.

\subsection{Recursive Feasibility}
\label{subsec:recursive_feasibility}
Recursive feasibility ensures the MPC algorithm never encounters an unsatisfiable constraint. To establish nominal guarantees, we introduce the following consistency assumption.

\begin{assump}[Nominal prediction] \label{assum:forecast_consistency}
The prediction model $\Psi$ in~\eqref{eq:MPC-dynamics} exactly represents the plant dynamics~\eqref{eq:SIQR-continuous}, and the transmission forecast matches the true parameters over the horizon, i.e., $\hat{\beta}_{m,j} = \beta(t_m + j\Delta t)$ for all $j = 0, \ldots, H-1$.
\end{assump}

\begin{theorem}[Recursive Feasibility]\label{thm:rec_feasibility}
Let Ass.~\ref{assum:piecewise_inputs}--\ref{assum:forecast_consistency} hold. If the MPC Problem~\eqref{eq:MPC-problem} is feasible at time $t_m$, then it remains feasible at time $t_{m+1} = t_m + \Delta t$.
\end{theorem}
\begin{proof}
Let $\{q^*_{m,j}\}_{j=0}^{H-1}$ and 
$\{\hat{x}^*_{m,j}\}_{j=0}^{H}$ be the optimal control and 
state sequences at $t_m$, satisfying feasibility of~\eqref{eq:MPC-problem}. By Ass.~\ref{assum:forecast_consistency}, 
$x(t_{m+1}) = \hat{x}^*_{m,1}$. Construct the candidate 
sequence for $t_{m+1}$:
\[
  \tilde{q}_{m+1,j} \coloneqq
  \begin{cases}
    q^*_{m,j+1}, & j = 0, \ldots, H{-}2, \\
    q_{\text{safe}}, & j = H{-}1.
  \end{cases}
\]
By forecast consistency, the candidate trajectory satisfies $\tilde{x}_{m+1,j} = \hat{x}^*_{m,j+1}$ for $j < H$, meaning $\tilde{x}_{m+1,H-1} = \hat{x}^*_{m,H} \in \mathcal{X}_f$. We verify candidate feasibility step-by-step. First, the shifted inputs and $q_{\text{safe}}$ satisfy the box constraints by definition. Second, for $j < H-1$, the spectral constraint holds because it coincides perfectly with step $j+1$ of the previously feasible problem. Third, for $j = H-1$, since $\tilde{x}_{m+1,H-1} \in \mathcal{X}_f$, the terminal set definition and the monotonicity of $M$ in $\beta$ (Sec.~\ref{subsec:spectral_decay}) guarantee $\lambda_{\max}(M(\tilde{s}_{m+1,H-1}, q_{\text{safe}} \mid \hat{\beta}_{m+1,H-1})) \leq \lambda_{\max}(M(\tilde{s}_{m+1,H-1}, q_{\text{safe}} \mid \beta_{\max})) \leq -\alpha$. Finally, Lemma~\ref{lem:terminal_invariance} ensures $\tilde{x}_{m+1,H} \in \mathcal{X}_f$.
\end{proof}

\begin{remark}[Scope of the nominal hypothesis]
\label{rem:nominal_vs_robust}
Ass.~\ref{assum:forecast_consistency} is a standard nominal hypothesis in MPC stability theory, ensuring the shifted candidate sequence remains feasible. Crucially, this strong assumption is required \emph{only} for the forward-looking recursive argument, not for the safety of the currently applied step. Because the applied control $q^*_{m,0}$ relies strictly on the measured state $s(t_m)$ and parameter $\beta_m$, the unconditional Layer (i) guarantee ensures inter-sample decay independently of future forecast mismatch. In practice, the MPC routinely recalibrates $\Psi$ and $\hat{\beta}$ at each step; while this may perturb nominal recursive feasibility, it never retroactively invalidates the safety of the applied step. To formally bridge this gap, Sec.~\ref{subsec:robustness} provides a robust counterpart in which the nominal constraints are replaced by upper-envelope versions, recovering recursive feasibility and finite-horizon realized decay under bounded forecast error, alongside the
structural robustness already provided by the $\beta_{\max}$ terminal set.
\end{remark}

\subsection{Closed-Loop Stability} \label{subsec:stability}
We establish stability in two layers: finite-horizon decay via the spectral constraint (Theorem~\ref{thm:stability_finiteT}), and global convergence via a feasible backup continuation (Theorem~\ref{thm:global_exp_decay}).

\begin{theorem}[Finite stability under the MPC closed loop]
\label{thm:stability_finiteT}
Let Ass.~\ref{assum:piecewise_inputs}--\ref{assum:forecast_consistency}
hold and suppose Problem~\eqref{eq:MPC-problem} is feasible at $t_0$.
Then $\forall T > t_0$, the closed-loop infected state satisfies
\begin{equation}\label{eq:cl_decay_finiteT}
  \|y(t)\|_1 \leq C_T\, e^{-\alpha(t-t_0)}\, \|y(t_0)\|_1,
  \qquad \forall\, t \in [t_0, T],
\end{equation}
with $C_T = (\bar{v}_T/\underline{v}_T)^{M_{\mathrm{tot}}+1}$ as in Prop.~\ref{prop:decay_eig_clean}, where $v_m$ is the normalized left Perron
eigenvector of the realized matrix
$M(s(t_m), q^*_{m,0}\mid\beta(t_m))$.
\end{theorem}

\begin{proof}
By Theorem~\ref{thm:rec_feasibility}, feasibility at $t_0$ implies
Problem~\eqref{eq:MPC-problem} is feasible at every sampling time
$t_m$, so the applied control satisfies
$\lambda_{\max}(M(s(t_m), q^*_{m,0} \mid \beta(t_m))) \leq -\alpha$.
Since $s(t) \leq s(t_m)$ and $q,\beta$ are constant on
$[t_m, t_{m+1})$, monotonicity of $M$ in $s$ and of $\lambda_{\max}$
for Metzler matrices~\cite{berman1994nonnegative} gives
$\lambda_{\max}(M(s(t), q(t) \mid \beta(t))) \leq -\alpha$ for all
$t \geq t_0$. Prop.~\ref{prop:decay_eig_clean} applied along the
closed-loop trajectory yields~\eqref{eq:cl_decay_finiteT}.
\end{proof}

\noindent\textbf{Global Exponential Decay Via Feasible Continuation.}
Theorem~\ref{thm:stability_finiteT} guarantees exponential decay on any finite interval, but $C_T$ may grow with $T$ (Remark~\ref{rem:CT_structure}). 
We complement this with a global bound by exhibiting a specific admissible continuation, distinct from the receding-horizon law, that decays globally at rate $\alpha$.

\begin{theorem}[Global exponential decay]
\label{thm:global_exp_decay}
Suppose Ass.~\ref{assum:piecewise_inputs}--\ref{assum:forecast_consistency} hold and Problem~\eqref{eq:MPC-problem} is feasible at $t_0$. Then there exists an admissible control signal and a constant $C > 0$, independent of time, such that 
\begin{equation}\label{eq:asymptotic_decay}
  \|y(t)\|_1 \leq C\, e^{-\alpha(t - t_0)}\, \|y(t_0)\|_1, \qquad \forall\, t \geq t_0.
\end{equation}
\end{theorem}

\begin{proof}
Feasibility at $t_0$ yields a sequence $\{q^*_j\}_{j=0}^{H-1}$ with predicted terminal state $\hat{x}_H^* \in \mathcal{X}_f$. Define the continuation policy: apply $q^*_j$ on $[t_0 + j\Delta t,\, t_0 + (j{+}1)\Delta t)$ for $j = 0, \ldots, H{-}1$, then $q_{\mathrm{safe}}$ for $t \geq t^* \coloneqq t_0 + H\Delta t$.

\emph{Phase 1} ($t \in [t_0, t^*]$): Under Ass.~\ref{assum:forecast_consistency}, the realized trajectory matches the prediction. By Prop.~\ref{prop:decay_eig_clean} with fixed horizon $H$, $\|y(t)\|_1 \leq C_H\, e^{-\alpha(t - t_0)}\, \|y(t_0)\|_1$ on $[t_0, t^*]$, where $C_H = (\bar{v}_H/\underline{v}_H)^{H+1}$ is finite. Crucially, at the transition time $t^*$, the state continuity and exact prediction matching ensure the realized state $x(t^*)$ equals the predicted terminal state $\hat{x}_H^*$. Because feasibility at $t_0$ guarantees $\hat{x}_H^* \in \mathcal{X}_f$, transitioning to $q_{\mathrm{safe}}$ is strictly feasible and stabilizing.

\emph{Phase 2} ($t \geq t^*$): Define $M^* \coloneqq M(s(t^*), q_{\mathrm{safe}} \mid \beta_{\max})$, which is irreducible (Ass.~\ref{assum:coupling}). Let $v^* > 0$ be its normalized Perron left eigenvector. Monotonicity of $M$ gives $M(s(t), q_{\mathrm{safe}} \mid \beta(t)) \leq M^*$ for all $t \geq t^*$. The Lyapunov function $V^*(t) = (v^*)^\top y(t)$ satisfies $\dot{V}^* \leq -\alpha V^*$, yielding $\|y(t)\|_1 \leq C_\infty e^{-\alpha(t-t^*)} \|y(t^*)\|_1$ with $C_\infty = \max_i v^*_i/\min_i v^*_i$. Combining phases gives $C \coloneqq C_H C_\infty$ in~\eqref{eq:asymptotic_decay}.
\end{proof}

\subsection{Robustness to Forecast Error}
\label{subsec:robustness}

The recursive-feasibility and stability results of
Sec.~\ref{subsec:recursive_feasibility}--\ref{subsec:stability} are nominal MPC guarantees:
Ass.~\ref{assum:forecast_consistency} ensures that the realized state at
the next sampling time matches the first predicted state, so that the
shifted-sequence argument applies. We give a robust counterpart in
which the nominal spectral and terminal constraints are replaced by
conservative upper-envelope constraints.

\noindent\textbf{Robust MPC formulation.}
At iteration $m$, let
$\bar s_{m,j}\in\mathbb R^n_{>0}$ for $j=0,\ldots,H$, and
$\bar\beta_{m,j}\in\mathbb R^2_{>0}$ for $j=0,\ldots,H-1$, denote
component-wise upper envelopes for the susceptible state and
transmission parameters over the horizon:
\begin{align}
\label{eq:env_validity_s}
  s(t_{m,j}) &\le \bar s_{m,j}, & j&=0,\ldots,H,\\
\label{eq:env_validity_beta}
  \beta(t_{m,j}) &\le \bar\beta_{m,j}, & j&=0,\ldots,H-1,
\end{align}
component-wise, with the standing global bound
$\bar\beta_{m,j}\le \beta_{\max}$ from Sec.~\ref{subsec:network_dynamics}.
The robust MPC replaces the nominal spectral
constraint~\eqref{eq:MPC-eig-constraint} by
\begin{equation}
\label{eq:robust_spectral_constraint}
  \lambda_{\max}\!\bigl(M(\bar s_{m,j},q_{m,j}\mid \bar\beta_{m,j})\bigr)
  \le -\alpha,
  \quad j=0,\ldots,H-1,
\end{equation}
and the nominal terminal constraint $\hat{x}_H \;\in\; \mathcal{X}_f$ in~\eqref{eq:MPC-problem} by
\begin{equation}
\label{eq:robust_terminal_constraint}
  \lambda_{\max}\!\bigl(M(\bar s_{m,H},q_{\mathrm{safe}}\mid \beta_{\max})\bigr)
  \le -\alpha.
\end{equation}
The cost and prediction dynamics~\eqref{eq:MPC-dynamics} are
unchanged. Because $M$ is entry-wise non-decreasing in $(s,\beta)$ and
$\lambda_{\max}$ is monotone on Metzler
matrices~\cite{berman1994nonnegative}, any
$q_{m,j}$ feasible
for~\eqref{eq:robust_spectral_constraint}--\eqref{eq:robust_terminal_constraint}
automatically satisfies $\lambda_{\max}(M(s(t_{m,j}),q_{m,j}\mid \beta(t_{m,j})))\le -\alpha$
under \eqref{eq:env_validity_s}--\eqref{eq:env_validity_beta}.

To establish the robust results, we propose the following assumption which is a strict relaxation of the nominal Ass.~\ref{assum:forecast_consistency}.
\begin{assump}[Shift-consistent envelopes]
\label{ass:shift_consistent}
The upper envelopes satisfy, for every $m\ge 0$,
\begin{align}
  \bar s_{m+1,j} &\le \bar s_{m,j+1},
    & j&=0,\ldots,H-1,
    \label{eq:shift_s}\\
  \bar\beta_{m+1,j} &\le \bar\beta_{m,j+1},
    & j&=0,\ldots,H-2,
    \label{eq:shift_beta}\\
  \bar s_{m+1,H} &\le \bar s_{m,H},
    \label{eq:shift_terminal}
\end{align}
component-wise, together with the standing global bound
$\bar\beta_{m,j}\le \beta_{\max}$.
\end{assump}

\begin{theorem}[Robust recursive feasibility]
\label{thm:robust_recursive_feasibility}
Let Ass.~\ref{assum:piecewise_inputs},~\ref{assum:coupling}, and~\ref{assum:discretization}
hold, and let the envelopes satisfy Ass.~\ref{ass:shift_consistent}.
If the robust MPC with
constraints~\eqref{eq:robust_spectral_constraint}--\eqref{eq:robust_terminal_constraint}
is feasible at $t_m$, it remains feasible at $t_{m+1}$.
\end{theorem}
\begin{proof}
Let $\{q^*_{m,j}\}_{j=0}^{H-1}$ be feasible at $t_m$. Define the
shifted candidate at $t_{m+1}$:
\[
  \tilde q_{m+1,j}=
  \begin{cases} q^*_{m,j+1}, & j=0,\ldots,H-2,\\
                q_{\mathrm{safe}}, & j=H-1.\end{cases}
\]
Box constraints are immediate.

\emph{Stages $j=0,\ldots,H-2$.} By \eqref{eq:shift_s}, \eqref{eq:shift_beta}, and
entry-wise monotonicity of $M$ in $(s,\beta)$,
\[
  M(\bar s_{m+1,j},\tilde q_{m+1,j}\mid \bar\beta_{m+1,j})
  \le M(\bar s_{m,j+1},q^*_{m,j+1}\mid \bar\beta_{m,j+1})
\]
entry-wise; both matrices are Metzler, so by \eqref{eq:robust_spectral_constraint} at step $j+1$ of iteration $m$ and the monotonicity of $\lambda_{\max}$
\[
  \lambda_{\max}\!\bigl(M(\bar s_{m+1,j},\tilde q_{m+1,j}\mid \bar\beta_{m+1,j})\bigr)
  \le -\alpha.
\]

\emph{Stage $j=H-1$.} Here $\tilde q_{m+1,H-1}=q_{\mathrm{safe}}$. By
\eqref{eq:shift_s} at $j=H-1$ and the standing bound
$\bar\beta_{m+1,H-1}\le \beta_{\max}$,
\[
  M(\bar s_{m+1,H-1},q_{\mathrm{safe}}\mid \bar\beta_{m+1,H-1})
  \le M(\bar s_{m,H},q_{\mathrm{safe}}\mid \beta_{\max})
\]
entry-wise, so
\[
  \lambda_{\max}\!\bigl(M(\bar s_{m+1,H-1},q_{\mathrm{safe}}\mid \bar\beta_{m+1,H-1})\bigr)\le -\alpha,
\]
by~\eqref{eq:robust_terminal_constraint} at $t_m$.

\emph{Terminal stage.} By \eqref{eq:shift_terminal} and monotonicity,
\begin{align*}
  &\lambda_{\max}\!\bigl(M(\bar s_{m+1,H},q_{\mathrm{safe}}\mid \beta_{\max})\bigr)
  \le  \\ 
  &\lambda_{\max}\!\bigl(M(\bar s_{m,H},q_{\mathrm{safe}}\mid \beta_{\max})\bigr)\le -\alpha,
\end{align*}
so~\eqref{eq:robust_terminal_constraint} holds at $t_{m+1}$. The shifted
candidate is feasible.
\end{proof}

\noindent\textbf{Realized closed-loop decay.}

\begin{corollary}[Robust realized decay]
\label{cor:robust_realized_decay}
Under the hypotheses of Theorem~\ref{thm:robust_recursive_feasibility},
suppose the robust MPC is feasible at $t_0$. Then for every $T>t_0$,
the realized closed-loop infected state satisfies
\begin{equation}
\label{eq:realized_decay}
  \|y(t)\|_1 \le C_T\, e^{-\alpha(t-t_0)} \|y(t_0)\|_1,
  \qquad t\in[t_0,T],
\end{equation}
with $C_T = (\bar{v}_T/\underline{v}_T)^{M_{\mathrm{tot}}+1}$ as in
Prop.~\ref{prop:decay_eig_clean}, where $v_m$ is the normalized left Perron eigenvector of the realized matrix
$M(s(t_m), q^*_{m,0}\mid\beta(t_m))$.
\end{corollary}
\begin{proof}
At every $t_m\in[t_0,T]$, Theorem~\ref{thm:robust_recursive_feasibility}
gives a feasible $q^*_{m,0}$ satisfying the robust spectral constraint
at $j=0$. Envelope
validity~\eqref{eq:env_validity_s}--\eqref{eq:env_validity_beta} and
monotonicity of $M$ and $\lambda_{\max}$ on Metzler matrices then yield $\lambda_{\max}(M(s(t_m),q^*_{m,0}\mid \beta(t_m)))\le -\alpha$. The inter-sample argument of Theorem~\ref{thm:stability_finiteT} with piecewise-constant $(q,\beta)$ and Prop.~\ref{prop:decay_eig_clean} gives~\eqref{eq:realized_decay}.
\end{proof}

\begin{remark}[On global decay under forecast error]
\label{rem:robust_global}
Unlike Theorem~\ref{thm:global_exp_decay}, the robust formulation does not
directly yield a time-uniform global decay bound. The obstruction is that
Theorem~\ref{thm:global_exp_decay} uses Ass.~\ref{assum:forecast_consistency}
to identify the realized state at $t^* = t_0 + H\Delta t$ with the
predicted terminal state $\hat x_H^* \in \mathcal X_f$, so that Phase~2's Lyapunov bound can be invoked from $t^*$. Under envelope dominance alone, the realized $x(t^*)$ need not lie in $\mathcal X_f$. Two structural properties of the closed loop nonetheless point toward a global guarantee. First, $\mathcal X_f$ is downward-closed in $s$: if $x \in \mathcal X_f$ and $s' \le s(x)$ component-wise, then any state with susceptible component $s'$ also lies in $\mathcal X_f$.
Second, susceptibles are monotonically non-increasing along any realized trajectory. Together these imply that once the realized
$s(t)$ enters $\mathcal X_f$, it remains in $\mathcal X_f$, and Phase~2 of
Theorem~\ref{thm:global_exp_decay} extends from that point. A precise finite-time entry condition, however, depends on initial conditions and network parameters, and we leave its characterization to future work. The robust guarantee in this paper is therefore stated on arbitrary finite horizons.
\end{remark}
\noindent\textbf{Nominal reduction.}
In the limit of exact prediction and tight envelopes,
$\bar s_{m,j}=\hat s_{m,j}=s(t_{m,j})$ and
$\bar\beta_{m,j}=\hat\beta_{m,j}=\beta(t_{m,j})$, the robust
constraints~\eqref{eq:robust_spectral_constraint}--\eqref{eq:robust_terminal_constraint}
reduce to the nominal
constraints~\eqref{eq:MPC-eig-constraint}--\eqref{eq:MPC-terminal},
Ass.~\ref{ass:shift_consistent} holds with equality under
Ass.~\ref{assum:forecast_consistency}, and
Theorem~\ref{thm:robust_recursive_feasibility} and
Corollary~\ref{cor:robust_realized_decay} recover
Theorem~\ref{thm:rec_feasibility} and Theorem~\ref{thm:stability_finiteT}.

\section{NUMERICAL VALIDATION}
\label{sec:numerical_results}

\subsection{Simulation Setup}
\noindent\textbf{Data, Calibration \& Parameters.} 
We validate the framework on a network of $n=14$ Massachusetts counties. Topology, populations $N_i$, and mobility $\tau_{ij}$ follow~\cite{talaei2025network}. Disease parameters are $r^{\mathrm{a}} = r^{\mathrm{s}} = 0.2$ and $\epsilon = 0.32$. Baseline transmission ($\beta^{\mathrm{a}} = 0.67\beta^{\mathrm{s}}$) is calibrated to the April 2020 Massachusetts epidemic growth~\cite{birge2022controlling}. The MPC uses a sampling interval $\Delta t = 7$ days, horizon $H=7$ weeks, target decay $\alpha = 0.023\,\text{day}^{-1}$, control weight $\rho = 0.1$, bound $B=2$, and terminal penalty $P = \mathrm{diag}(0_n, I_n, I_n, 0_n)$. To guarantee integration accuracy over the 7-day prediction and control steps, both the MPC internal trajectories and the closed-loop plant simulation are integrated via an adaptive Runge--Kutta solver (\texttt{ode45}, $\mathrm{RelTol}=10^{-8}$) with internal substep selection.

\noindent\textbf{Scenarios.}
We evaluate three settings: \emph{Pure} (no limits); \emph{Smoothing Cost} (penalizes increases in isolation rates with $\rho_{\text{smooth}}=1$); and \emph{Rate Constraints} ($\Delta q_i(t)\le 0.2$/week, modeling administrative inertia). In all cases, a variant emerges at day 28, increasing transmission by 80\%. We compare MPC against a prediction-free myopic baseline that minimizes instantaneous cost subject to $\lambda_{\max}(M)\le -\alpha$. If rate limits render the myopic problem infeasible during the shock, it applies the maximum admissible update, temporarily violating the decay constraint.

\begin{figure*}[t] 
    \centering
    
    \subfloat[Infection prevalence ($\|y\|_1$).\label{fig:infections}]{%
        \includegraphics[width=0.32\textwidth]{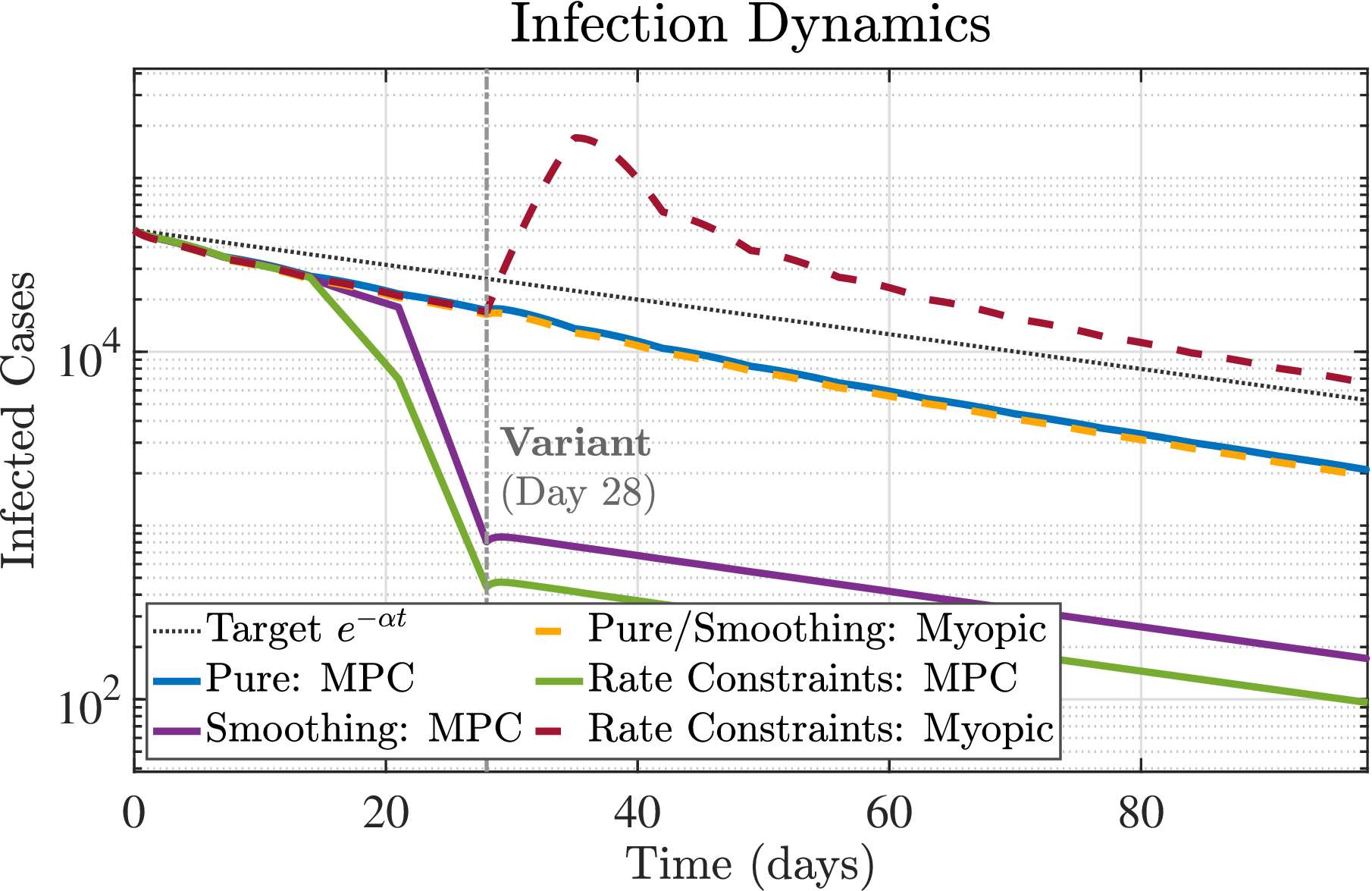}%
    }
    \hfil 
    \subfloat[Average control intensity ($\bar{q}$).\label{fig:controls}]{%
        \includegraphics[width=0.32\textwidth]{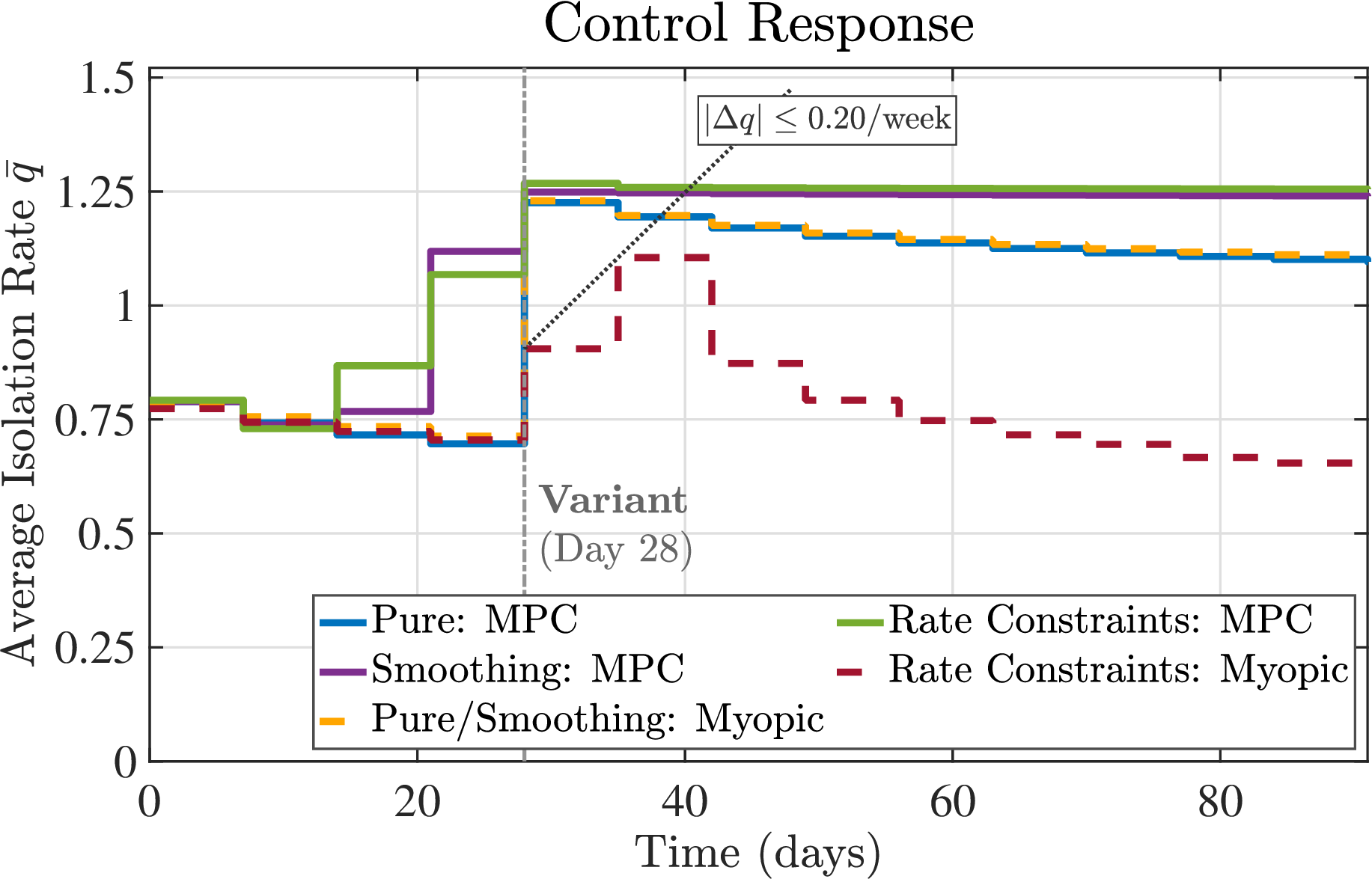}%
    }
    \hfil
    \subfloat[Total isolated population.\label{fig:quarantine}]{%
        \includegraphics[width=0.32\textwidth]{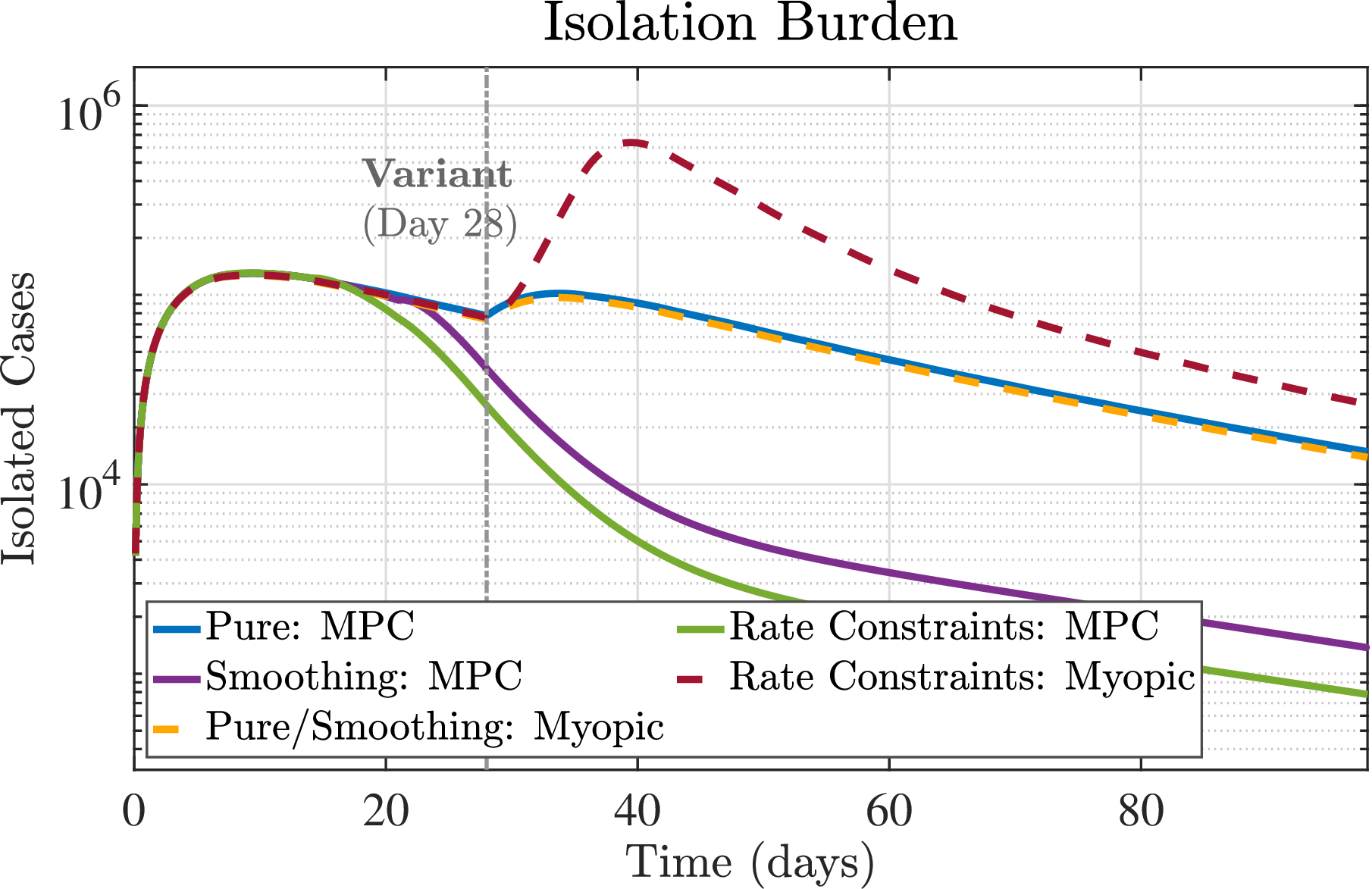}%
    }
    
    \caption{\textbf{Performance Across the Three Settings.} 
    (a) Infection dynamics: MPC maintains exponential decay (dotted line shows target $e^{-\alpha t}$) across all scenarios, while Myopic fails under rate constraints. 
    (b) Control response 
    ($\bar q(t) \coloneqq \tfrac{1}{2n}\sum_{i=1}^{n}\bigl(q^{\mathrm{a}}_i(t)+q^{\mathrm{s}}_i(t)\bigr)$): MPC anticipates the day-28 variant shock and ramps up preemptively, whereas Myopic reacts only at day 28. 
    Under rate constraints, Myopic cannot ramp up fast enough and saturates at maximum allowed rate.
    (c) Isolation burden.}
    \label{fig:simulation_results}
\end{figure*}

\subsection{Computational Approach and Tractability}
\label{subsec:computation}

The non-convex NLP~\eqref{eq:MPC-problem} contains $2nH$ decision variables and is solved in MATLAB using \texttt{fmincon} (interior-point), evaluating the spectral constraint via the largest real eigenvalue of $M(\hat s_j, q_j \mid \hat\beta_{m,j})$ at each step. Because the core contribution is the spectral-certificate formulation itself, any standard NLP tool (e.g., IPOPT, KNITRO via CasADi) can be substituted without affecting the safety guarantees of Theorems~\ref{thm:rec_feasibility}--\ref{thm:global_exp_decay}. At $t_0$, the solver is warm-started with a constant sequence $q_j = \bar q_0\mathbf{1}_{2n}$, where $\bar q_0$ is the smallest control satisfying the spectral constraint at $x_0$ (found via bisection). At subsequent steps, we use the shifted sequence $\{q^*_{m,1},\dots,q^*_{m,H-1},q_{\mathrm{safe}}\}$, which is guaranteed recursively feasible by Theorem~\ref{thm:rec_feasibility}.

\noindent\textbf{Tractability.}
For the $n=14$, $H=7$ baseline (196 variables), \texttt{fmincon} successfully converged to a first-order optimal point at every step across all scenarios, strictly satisfying the realized target $\lambda_{\max}(M)\le -0.023$\,day$^{-1}$ (verifying Layer~1). Median time to reach a solution was $58$\,s (max $129$\,s) on an Apple M2 Pro; the shifted warm-start reduced median iterations from 70 (at $t_0$) to 49. Because the NLP is non-convex, global optimality is not guaranteed; however, recursive feasibility strictly ensures any returned local optimum safely satisfies constraints~\eqref{eq:MPC-eig-constraint}--\eqref{eq:MPC-terminal}. Solution times scale approximately quadratically with the horizon due to $\mathcal{O}(H)$ ODE and eigenvalue evaluations, with median times of $(1.9, 11.7, 30.8, 58.1, 130.1)$\,s for $H\in\{1,3,5,7,10\}$. Given the weekly sampling period ($\Delta t = 7$\,days), these sub-two-minute iterations easily support real-time deployment.

\subsection{Performance Analysis}
Fig.~\ref{fig:simulation_results} summarizes the dynamics across the three scenarios. In the \emph{Pure} and \emph{Smoothing Cost} settings, both controllers easily maintain exponential decay since policies can adjust instantaneously. However, while Myopic ignores the smoothing penalty (optimizing only one step), MPC exploits the horizon to gradually ramp up interventions before the variant emerges, softening the abruptness of lockdown shifts. 

The critical divergence occurs under \emph{Rate Constraints} ($\Delta q_i \le 0.2$/week). Myopic cannot react fast enough to the 80\% transmission surge, losing the decay guarantee and allowing prevalence to peak at roughly $1.7\times 10^5$ cases. Conversely, MPC anticipates the shock and begins tightening control two weeks early. This preemptive accumulation of control authority absorbs the surge while respecting rate limits. Although this slightly increases pre-shock interventions, it strictly preserves exponential decay and reduces the cumulative isolation burden by $>50\%$ relative to Myopic. Thus, prediction becomes strictly essential under realistic implementation delays.

\subsection{Attribution: Hard vs.\ Soft Spectral Enforcement} 
To verify that this preemptive behaviour is specific to the hard spectral certificate rather than a generic consequence of having a prediction horizon, we ran an ablation in which constraint~\eqref{eq:MPC-eig-constraint} is replaced by a quadratic 
penalty $\tfrac{1}{2}\rho_\lambda\sum_{j=0}^{H-1}(\lambda_{\max}(\widehat M_j)+\alpha)_+^2$ in the cost, with the horizon, forecast, prediction model, terminal constraint, and rate limit all held fixed.

\noindent\textbf{Calibration.} A one-step sweep at $t_0$ under unlimited rate changes shows that the soft variant matches the hard solution only as $\rho_\lambda$ grows: realized spectral abscissa 
$\{+0.258,\,+0.020,\,-0.019,\,-0.023\}$\,day$^{-1}$ at $\rho_\lambda\in\{10,10^2,10^3,10^4\}$, against $-0.023$ for the hard variant. This convergence behaviour is consistent with classical exterior-penalty theory. As $\rho_\lambda\to\infty$, the 
penalized optimum approaches the hard-constrained optimum~\cite{bertsekas1997nonlinear}. Two practical limitations accompany this convergence. First, \texttt{fmincon}'s interior-point iteration count at the initial step grew from $131$ at $\rho_\lambda=10^2$ to 
$324$ at $\rho_\lambda=10^3$, and exceeded $500$ iterations at $\rho_\lambda=10^4$. Second, the penalty has a discontinuous second derivative at the constraint boundary; this lack of smoothness compounds with the penalty's growing magnitude as $\rho_\lambda$ increases. Recovering hard-constraint behaviour therefore requires 
$\rho_\lambda$ in a regime where the formulation is no longer a meaningfully soft alternative but a numerically degraded relaxation of the hard constraint.

\noindent\textbf{Closed-loop result under rate limits.} With a representative moderate weight $\rho_\lambda=10^2$, the soft MPC fails to ramp up preemptively, loses the decay guarantee at $13$ of $14$ sampling instants, and reaches a peak realized spectral abscissa of $+0.059$\,day$^{-1}$ at variant onset---qualitatively the same failure mode as the Myopic baseline. Fig.~\ref{fig:soft_spectral} visualizes this directly: panel~(a) shows infection trajectories, panel~(b) the absence of preemptive control ramp-up, and panel~(c) the realized $\lambda_{\max}$ along the closed loop, where the soft variant sits above the safety threshold $-\alpha$ for nearly the entire horizon.

The ablation also clarifies the relationship to state-dependent safety constraints in epidemic MPC. ICU- and hospitalization-cap formulations~\cite{kohler2021robust,carli2020model} become binding only as the relevant compartment approaches its threshold; in our scenario, prevalence at the variant onset is well below typical hospital-cap thresholds, so a cap-based controller would not trigger preemptive action and would react only as infections rise---precisely the failure mode of both the Myopic baseline and the soft-spectral variant. The anticipatory ramp-up is thus a consequence of the time-coupled feasibility of a state-independent spectral certificate enforced as a hard constraint, not a generic property of MPC with prediction or of smoothly penalizing the certificate at moderate weight.

\begin{figure*}[t]
    \centering
    \subfloat[Infection prevalence ($\|y\|_1$).\label{fig:soft_inf}]{%
        \includegraphics[width=0.32\textwidth]{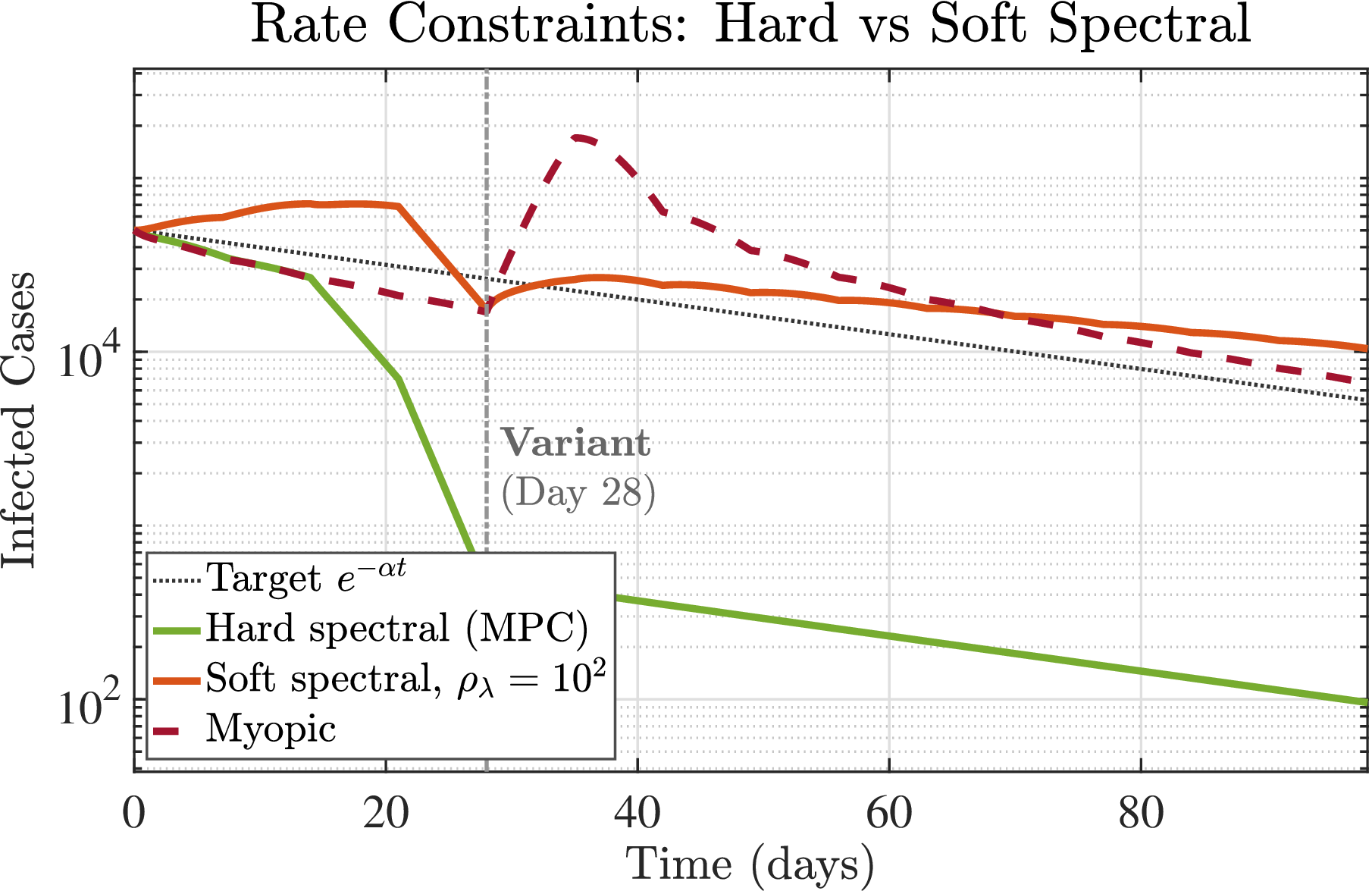}}%
    \hfil
    \subfloat[Average control intensity ($\bar{q}$).\label{fig:soft_q}]{%
        \includegraphics[width=0.32\textwidth]{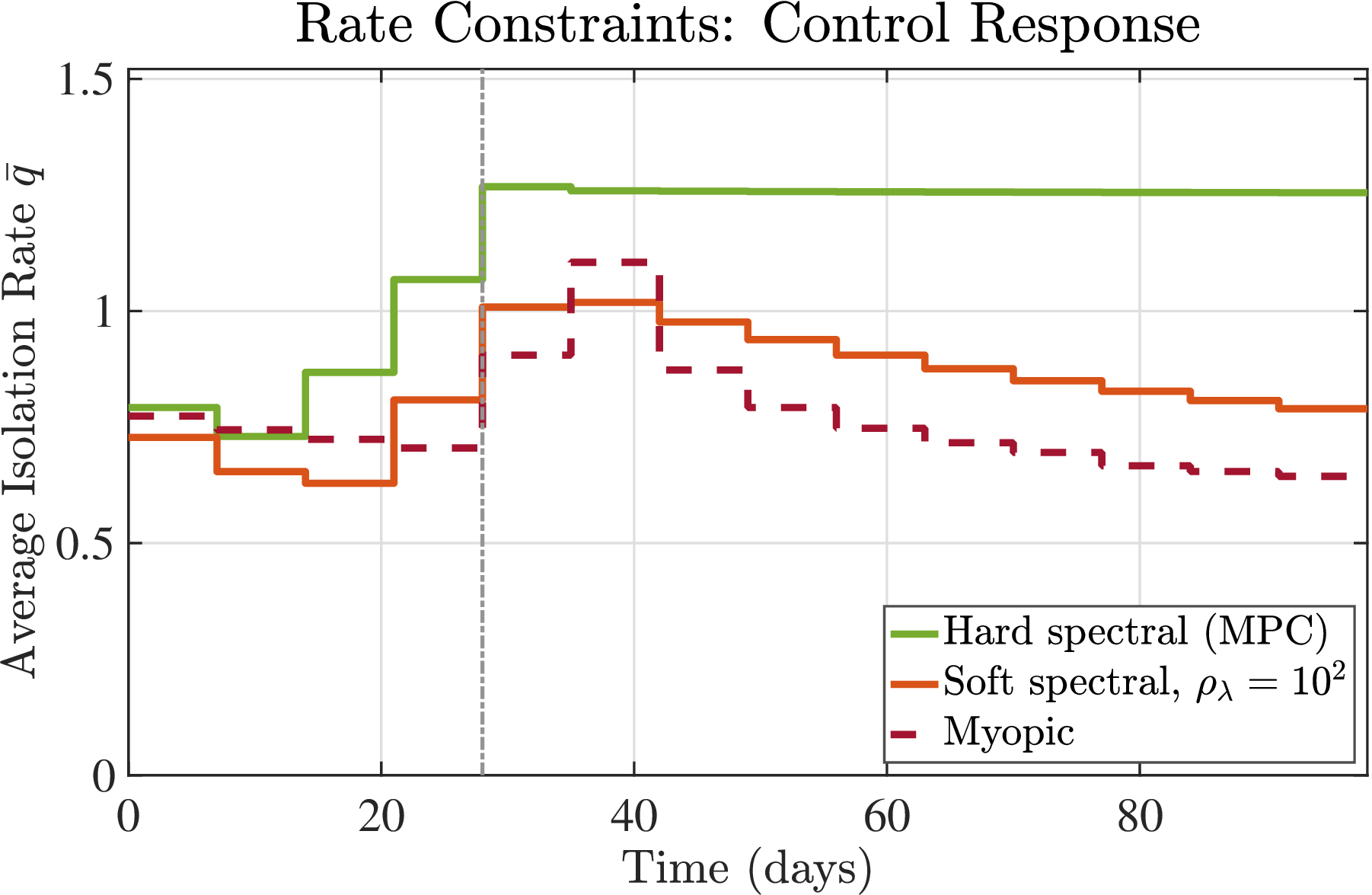}}%
    \hfil
    \subfloat[Realized $\lambda_{\max}(M)$.\label{fig:soft_lambda}]{%
        \includegraphics[width=0.32\textwidth]{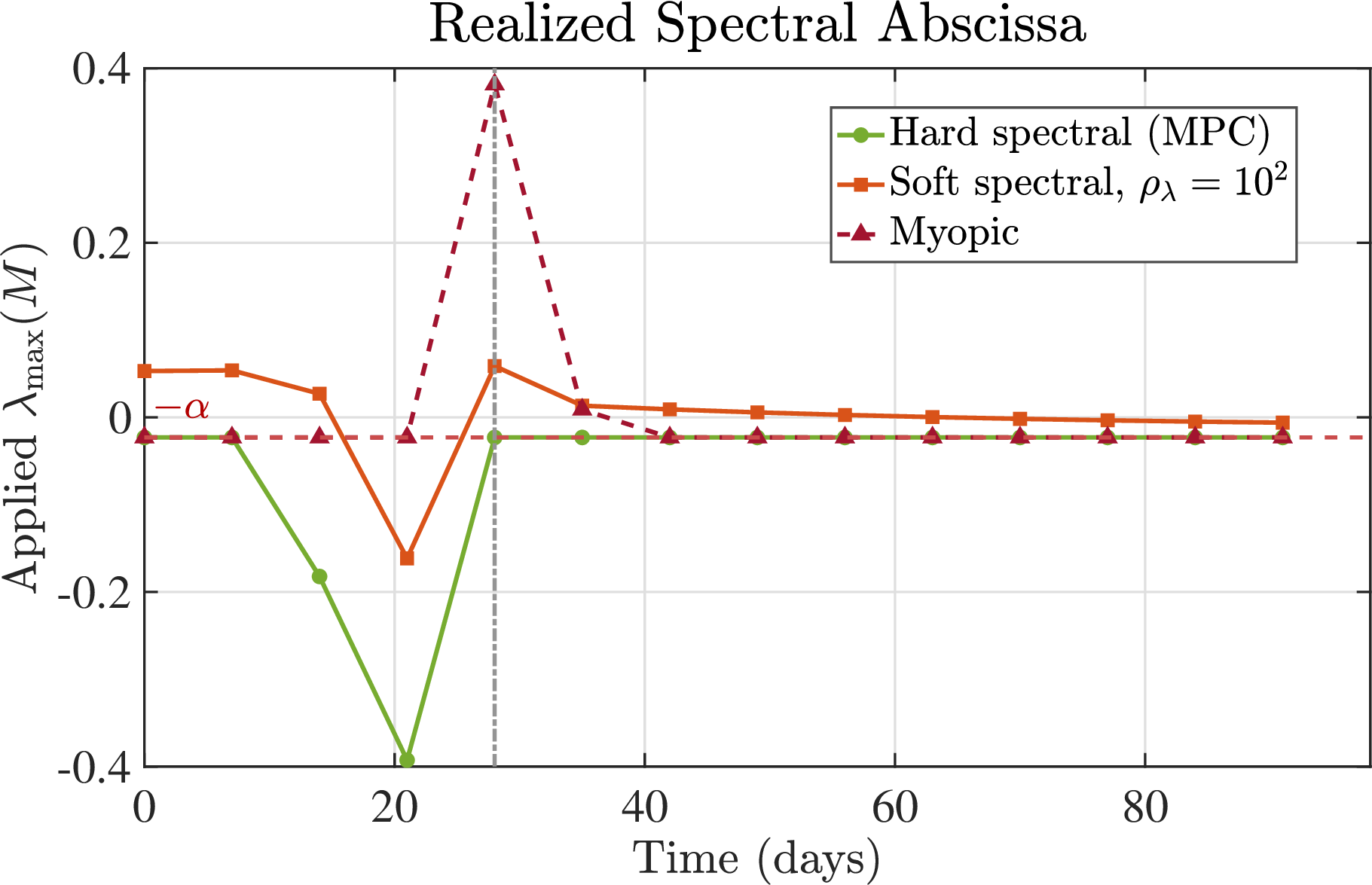}}%
    \caption{\textbf{Soft-spectral ablation under Rate Constraints.} 
    Soft variant ($\rho_\lambda=10^2$) compared to the hard-spectral MPC and the Myopic baseline. (a) Infections: the soft variant drifts above target $e^{-\alpha t}$ and behaves like Myopic. (b) Controls: only the hard variant ramps up preemptively (around day~14); the soft variant ramps up reactively, starting only one week before onset (day~21). (c) Realized spectral abscissa: hard variant sits at the design target $-\alpha$ throughout, while the soft variant and Myopic both violate the safety threshold for most 
    of the horizon.}
    \label{fig:soft_spectral}
\end{figure*}

\section{CONCLUSIONS}
We presented a certificate-based MPC framework for networked epidemic management that embeds a spectral decay condition as a hard stage-wise constraint, with terminal-set invariance following from monotone susceptible depletion. Under nominal prediction we established recursive feasibility, finite-horizon closed-loop decay, and the existence of a feasible continuation achieving global exponential decay; a robust counterpart recovers recursive feasibility and finite-horizon realized decay under a strict relaxation of the exact-prediction hypothesis. On the Massachusetts network, the spectral constraint produced preemptive isolation ramp-up under administrative rate limits, where myopic and state-dependent formulations fail.

\noindent\textbf{Limitations \& future work.} The spectral constraint is a global property of the network and does not admit straightforward decomposition; distributed implementations via dual decomposition or ADMM on the spectral certificate are a natural direction for future work. The non-convexity of the NLP precludes global optimality guarantees, though recursive feasibility ensures safety along any first-order optimal trajectory. Finally, while standard continuous-time safety-critical designs (e.g., control barrier functions~\cite{ames2017control}) are myopic and struggle with rate limits, extending predictive safety-critical tools to handle non-smooth eigenvalue constraints remains a compelling theoretical direction.


\bibliographystyle{IEEEtran}   
\bibliography{refs}            

@article{talaei2025network,
  title={Network-based epidemic control through optimal travel and quarantine management},
  author={Talaei, Mahtab and Rikos, Apostolos I and Olshevsky, Alex and White, Laura F and Paschalidis, Ioannis Ch},
  journal={IEEE Transactions on Control of Network Systems},
  year={2025},
  publisher={IEEE}
}

@article{ma2023optimal,
  title={Optimal fixed lockdown for pandemic control},
  author={Ma, Qianqian and Liu, Yang-Yu and Olshevsky, Alex},
  journal={IEEE Transactions on Automatic Control},
  volume={69},
  number={7},
  pages={4538--4553},
  year={2023},
  publisher={IEEE}
}

@book{berman1994nonnegative,
  title={Nonnegative matrices in the mathematical sciences},
  author={Berman, Abraham and Plemmons, Robert J},
  year={1994},
  publisher={SIAM}
}

@article{1981:Cohen,
  title={Convexity of the dominant eigenvalue of an essentially nonnegative matrix},
  author={Cohen, Joel E},
  journal={Proceedings of the American Mathematical Society},
  volume={81},
  number={4},
  pages={657--658},
  year={1981}
}

@article{carli2020model,
  title={Model predictive control to mitigate the COVID-19 outbreak in a multi-region scenario},
  author={Carli, Raffaele and Cavone, Graziana and Epicoco, Nicola and Scarabaggio, Paolo and Dotoli, Mariagrazia},
  journal={Annual Reviews in Control},
  volume={50},
  pages={373--393},
  year={2020},
  publisher={Elsevier}
}

@article{kohler2021robust,
  title={Robust and optimal predictive control of the {COVID-19} outbreak},
  author={K{\"o}hler, Johannes and Schwenkel, Lukas and Koch, Anne and Berberich, Julian and Pauli, Patricia and Allg{\"o}wer, Frank},
  journal={Annual Reviews in Control},
  volume={51},
  pages={525--539},
  year={2021},
  publisher={Elsevier}
}

@inproceedings{herceg2025scenario,
  title={A Scenario-based Model Predictive Control Scheme for Pandemic Response through Non-pharmaceutical Interventions},
  author={Herceg, Domagoj and Dell'Oro, Marco and Bertollo, Riccardo and Miura, Fuminari and de Klaver, Paul and Breschi, Valentina and Krishnamoorthy, Dinesh and Salazar, Mauro},
  booktitle={2025 IEEE Conference on Control Technology and Applications (CCTA)},
  pages={139--144},
  year={2025},
  organization={IEEE}
}

@article{watkins2019robust,
  title={Robust economic model predictive control of continuous-time epidemic processes},
  author={Watkins, Nicholas J and Nowzari, Cameron and Pappas, George J},
  journal={IEEE Transactions on Automatic Control},
  volume={65},
  number={3},
  pages={1116--1131},
  year={2019},
  publisher={IEEE}
}

@article{nowzari2016analysis,
  title={Analysis and control of epidemics: A survey of spreading processes on complex networks},
  author={Nowzari, Cameron and Preciado, Victor M and Pappas, George J},
  journal={IEEE Control Systems Magazine},
  volume={36},
  number={1},
  pages={26--46},
  year={2016},
  publisher={IEEE}
}

@book{lenhart2007optimal,
  title={Optimal control applied to biological models},
  author={Lenhart, Suzanne and Workman, John T},
  year={2007},
  publisher={Chapman and Hall/CRC}
}

@article{flaxman2020estimating,
  title={Estimating the effects of non-pharmaceutical interventions on {COVID-19} in {Europe}},
  author={Flaxman, Seth and Mishra, Swapnil and Gandy, Axel and Unwin, H Juliette T and Mellan, Thomas A and Coupland, Helen and Whittaker, Charles and Zhu, Harrison and Berah, Tresnia and Eaton, Jeffrey W and others},
  journal={Nature},
  volume={584},
  number={7820},
  pages={257--261},
  year={2020},
  publisher={Nature Publishing Group UK London}
}

@article{hsiang2020effect,
  title={The effect of large-scale anti-contagion policies on the {COVID-19} pandemic},
  author={Hsiang, Solomon and Allen, Daniel and Annan-Phan, S{\'e}bastien and Bell, Kendon and Bolliger, Ian and Chong, Trinetta and Druckenmiller, Hannah and Huang, Luna Yue and Hultgren, Andrew and Krasovich, Emma and others},
  journal={Nature},
  volume={584},
  number={7820},
  pages={262--267},
  year={2020},
  publisher={Nature Publishing Group UK London}
}

@article{thunstrom2020benefits,
  title={The benefits and costs of using social distancing to flatten the curve for {COVID-19}},
  author={Thunstr{\"o}m, Linda and Newbold, Stephen C and Finnoff, David and Ashworth, Madison and Shogren, Jason F},
  journal={Journal of Benefit-Cost Analysis},
  volume={11},
  number={2},
  pages={179--195},
  year={2020},
  publisher={Cambridge University Press}
}

@article{davies2021estimated,
  title={Estimated transmissibility and impact of {SARS-CoV-2} lineage B. 1.1. 7 in {England}},
  author={Davies, Nicholas G and Abbott, Sam and Barnard, Rosanna C and Jarvis, Christopher I and Kucharski, Adam J and Munday, James D and Pearson, Carl AB and Russell, Timothy W and Tully, Damien C and Washburne, Alex D and others},
  journal={Science},
  volume={372},
  OPTnumber={6538},
  OPTpages={eabg3055},
  year={2021},
  publisher={American Association for the Advancement of Science}
}

@article{barnett2023epidemic,
  title={Epidemic responses under uncertainty},
  author={Barnett, Michael and Buchak, Greg and Yannelis, Constantine},
  journal={Proceedings of the National Academy of Sciences},
  volume={120},
  number={2},
  pages={e2208111120},
  year={2023},
  publisher={National Academy of Sciences}
}

@article{rawlings2017model,
  title={Model predictive control with discrete actuators: Theory and application},
  author={Rawlings, James B and Risbeck, Michael J},
  journal={Automatica},
  volume={78},
  pages={258--265},
  year={2017},
  publisher={Elsevier}
}

@article{peni2020nonlinear,
  title={Nonlinear model predictive control with logic constraints for {COVID-19} management},
  author={P{\'e}ni, Tam{\'a}s and Csutak, Bal{\'a}zs and Szederk{\'e}nyi, G{\'a}bor and R{\"o}st, Gergely},
  journal={Nonlinear Dynamics},
  volume={102},
  number={4},
  pages={1965--1986},
  year={2020},
  publisher={Springer}
}

@article{armaou2022designing,
  title={Designing social distancing policies for the {COVID-19} pandemic: A probabilistic model predictive control approach},
  author={Armaou, Antonis and Katch, Bryce and Russo, Lucia and Siettos, Constantinos},
  journal={arXiv preprint arXiv:2202.00924},
  year={2022}
}

@article{birge2022controlling,
  title={Controlling epidemic spread: Reducing economic losses with targeted closures},
  author={Birge, John R and Candogan, Ozan and Feng, Yiding},
  journal={Management Science},
  volume={68},
  number={5},
  pages={3175--3195},
  year={2022},
  publisher={INFORMS}
}

@article{ames2017control,
  title={Control barrier function based quadratic programs for safety critical systems},
  author={Ames, Aaron D and Xu, Xiangru and Grizzle, Jessy W and Tabuada, Paulo},
  journal={IEEE Transactions on Automatic Control},
  volume={62},
  number={8},
  pages={3861--3876},
  year={2016},
  publisher={IEEE}
}

@article{mayne2000constrained,
  title={Constrained model predictive control: Stability and optimality},
  author={Mayne, David Q and Rawlings, James B and Rao, Christopher V and Scokaert, Pierre OM},
  journal={Automatica},
  volume={36},
  number={6},
  pages={789--814},
  year={2000},
  publisher={Elsevier}
}

@article{bertsekas1997nonlinear,
  title={Nonlinear programming},
  author={Bertsekas, Dimitri P},
  journal={Journal of the Operational Research Society},
  volume={48},
  number={3},
  pages={334--334},
  year={1997},
  publisher={Taylor \& Francis}
}

@article{esterhuizen2024mpc,
  title={MPC without terminal ingredients tailored to the SEIR compartmental epidemic model},
  author={Esterhuizen, Willem and Sauerteig, Philipp and Streif, Stefan and Worthmann, Karl},
  journal={Systems \& Control Letters},
  volume={193},
  pages={105908},
  year={2024},
  publisher={Elsevier}
}

@article{gemignani2026real,
  title={Real-time responses to epidemics: A Reinforcement-Learning approach},
  author={Gemignani, Gabriele and d'Onofrio, Alberto and Landi, Alberto and Pisaneschi, Giulio and Manfredi, Piero},
  journal={Mathematical Biosciences and Engineering},
  volume={23},
  number={3},
  pages={753--775},
  year={2026}
}

\end{document}